\documentclass[10pt]{amsart}
\usepackage{caption}
\usepackage{subcaption}
\usepackage{float}
\usepackage{enumerate,amssymb,  mathrsfs, graphicx}
\newtheorem{theorem}{Theorem}[section]
\newtheorem{lemma}[theorem]{Lemma}
\newtheorem*{lemma*}{Lemma}

\newtheorem{corollary}[theorem]{Corollary}
\theoremstyle{definition}

\newtheorem{example}[theorem]{Example}

\theoremstyle{remark}

\numberwithin{equation}{section}

\DeclareMathOperator{\diff}{d\!}

\begin{document}
\title{A sharp estimate of area for sublevel-set of Blaschke products}

\author[D. Kalaj]{David Kalaj}
\address{University of Montenegro, Faculty of natural sciences and mathematics, Podgorica, Cetinjski put b.b. 81000 Podgorica, Montenegro }
\email{davidk@ucg.ac.me}





\maketitle

\begin{abstract}
Let $\mathbb{D}$ be the unit disk in the complex plane. Among other results, we prove the following curious result for a finite Blaschke product:  $$B(z)=e ^{is}\prod_{k=1}^d \frac{z-a_k}{1-z \overline{a_k}}.$$ The Lebesgue measure of the sublevel set of $B$ satisfies the following sharp inequality for $t \in [0,1]$: $$|\{z\in \mathbb{D}:|B(z)|<t\}|\le \pi t^{2/d},$$ with equality at a single point $t\in(0,1)$ if and only if $a_k=0$ for every $k$. In that case the equality is attained for every $t$.
\end{abstract}

\maketitle
\section{Introduction}

In this paper, we prove some sharp estimates for the sublevel set and superlevel set of certain mappings having a log-subharmonic modulus. Sharp estimates for these sets have been substantial for the development of optimal estimates in the Faber-Krahn inequality in the works of various authors \cite{Nicola, ramostilli, kulikov2022, kalajramos2024, kalaj2024, Inv2, frank}. In most of those results, functions with compact support at the domain have been considered. In this paper, we consider a different case, where the functions do not have compact support.

\subsection{Quasiregular mappings}

Assume that $\Omega, D$ are domains in the complex plane.
For a smooth mapping $f: \Omega \to D$, we say that $f$ is $K$-quasiregular for $K \ge 1$ if
\begin{equation}\label{quasireg}
|f_{\bar z}| \le \frac{K-1}{K+1}|f_z|
\end{equation}
for $z \in \Omega$. If $f$ is an injection, then we call it $K$-quasi-conformal. If $f$ is quasiregular, then its Jacobian
\[ J(z, f) = |f_z|^2 - |f_{\bar z}|^2 \ge \frac{4K}{(1+K)^2}|f_z|^2 \ge 0. \]
So $f$ is sense-preserving. We say that a smooth mapping of the unit disk onto itself is $K$-quasiregular at the boundary if $f$ has a differentiable extension up to the boundary and the inequality \eqref{quasireg} is true for $z = e^{is}$ for almost every $s \in [0, 2\pi]$. In that case, for almost every $s$,
\begin{equation}\label{cohold}
|\partial_r f(z)| \le K|\partial_t f(z)|,
\end{equation}
where $z = e^{is}$.

\subsection{Index of a curve}

The integer $d$ is called the winding number of a closed piecewise smooth curve $\gamma: [0, 2\pi] \to \mathbb{C} \setminus \{w\}$ around $w$, representing the total number of times that the curve travels counterclockwise around $w$. That number is denoted by $\mathrm{Ind}_\gamma(w)$ and can be calculated by the formula
\[ \mathrm{Ind}_\gamma(w) = \frac{1}{2\pi i}\int_0^{2\pi} \frac{\gamma'(s)}{\gamma(s)-w} \, ds. \]
Assume that $\mathbb{D}$ is the unit disk and $\mathbf{T}$ is the unit circle.
If $f: \mathbb{D} \to \mathbb{D}$ is a smooth mapping having a smooth extension up to the boundary $\mathbf{T} = \partial \mathbb{D}$, then the curve $\gamma(s) = f(e^{is})$, $s \in [0, 2\pi]$, has a finite index around 0 or any other point $w \in \mathbb{D}$.

The main result of this paper is the following theorem:

\begin{theorem}\label{teo2}
Let $f$ be a real analytic mapping of the unit disk onto itself which is $K$-quasiregular at the boundary having a log-subharmonic modulus. For $t \in [0,1]$, let $A_t = \{z: |f(z)| \ge t\}$. Then
\[ \mu(t) = |A_t| \ge \pi - \pi t^{2/(Kd)}, \]
where $d = \mathrm{Ind}_{f(\mathbf{T})}(0)$ is the winding number of $f(\mathbf{T})$ around the origin. The equality is attained if and only if $f(z) = e^{ids} |z|^{Kd}$, $z=|z|e^{is}$.
\end{theorem}

We need the following simple lemma:

\begin{lemma}\label{ulu}
Let $f$ be a sense-preserving smooth mapping of the unit disk onto itself having a smooth extension up to the boundary so that $\lim_{|z| \to 1} |f(z)| = 1$. Then,
\[ \int_{|z|=1} |\partial_s f(z)| \, |dz| = 2\pi \mathrm{Ind}_{f(\mathbf{T})}(0) = 2\pi d. \]
Moreover, $f(e^{is}) = e^{i\phi(s)}$, where $\phi: [0, 2\pi] \to [\tau, \tau+2\pi d]$ is a non-decreasing surjection.
\end{lemma}

\begin{proof}[Proof of Lemma~\ref{ulu}]
First of all, if $z = re^{is}$, then $J(z, f) = \frac{1}{r}\Im(f_s(z) \overline{f_r(z)}) \ge 0$. Thus, for $z = e^{is}$, $f(e^{is}) = e^{i\phi(s)}$ for some smooth function $\phi: [0, 2\pi] \to \mathbb{R}$ so that $\phi(2\pi) = \phi(0) + 2k\pi$, $k \in \mathbb{Z}$, and thus
\begin{equation}\label{phip}
J(z, f) = \phi'(s)\Im(i f(z) \overline{f_r(z)}) = \phi'(s)\Re(f(z) \overline{f_r(z)}).
\end{equation}
Since
\[ \lim_{r \to 1} \frac{1 - |f(z)|^2}{1 - r} = 2\Re [f_r(e^{is}) \overline{f(e^{is})}] \ge 0, \]
in view of \eqref{phip}, we obtain that $\phi'(s) \ge 0$.

Now,
\[ |z| = 1 \Rightarrow |f(z)| = 1 \]
implies that
\[ |z| = 1 \Rightarrow \langle \partial_s f(z), f(z) \rangle = \Re (\partial_s f(z) \overline{f(z)}) = 0. \]
Thus,
\[ |z| = 1 \Rightarrow i \phi'(s) = \partial_s f(z) \overline{f(z)} = \pm i |\partial_s f(z)|. \]
Since $\phi'(s) \ge 0$, we have
\[ |z| = 1 \Rightarrow \frac{\partial_s f(z)}{f(z)} = i |\partial_s f(z)|. \]

Thus,
\[ 2\pi d i = i \int_{|z|=1} \phi'(s) \, dz = i \int_{|z|=1} |\partial_s f(z)| \, |dz|, \]
or what is the same,
\[ \int_{0}^{2\pi} \phi'(s) \, ds = 2 \pi d = \phi(2\pi) - \phi(0). \]
We conclude that $\phi: [0, 2\pi] \to [\tau, \tau + 2\pi d]$ is a non-decreasing function.
\end{proof}

\begin{proof}[Proof of Theorem~\ref{teo2}]
Let $u(z) = |f(z)|$ and $A_t = \{z: u(z) > t\}$, $0 < t < 1$.
We start with the co-area formula
\[ \mu(t) = |A_t| = \int_{A_t} dx = \int_t^{\max u} \int_{|u(z)|=\kappa} |\diff z| \, d\kappa. \]
Then we get
\begin{equation}\label{derlevel}
-\mu'(t) = \int_{\{z: u(z) = t\}} |\nabla u(z)|^{-1} |\diff z|
\end{equation}
along with the claim that $\{z: u(z) = t\} = \partial A_t \cap \mathbb{D}$ and that this set is a smooth curve for almost all $t \in (0, t_\circ)$.
These assertions follow the proof of \cite[Lemma~3.2]{Nicola}. We point out that, since $u$ is real analytic, it is a well-known fact from measure theory that the level set $\{z: u(z) = t\}$ has zero measure \cite{zero}, and this is equivalent to the fact that $\mu$ is continuous.

Following the approach from \cite{Nicola}, our next step is to apply the H\"older inequality with $p=q=2$.

 Then $$|\partial_\circ A_t|^2=\left(\int_{\partial_\circ A_t}|\diff z|\right)^2\le \int_{\partial_\circ A_t} |\nabla u|^{-1} |\diff z|\int_{\partial_\circ A_t} |\nabla u||\diff z|,$$ where $\partial_\circ A_t$ is part of the boundary $A_t$ inside of the unit disk which coincides with the level set $\{z: u(z)=t\}$, shortly written as $u=t$.
Further, let $\mu(t)=|A_t|$. Then
 $$|\partial_\circ A_t|^2\le \mu'(t)(-\int_{u=t} |\nabla u|)|\diff z|.$$
Further by isoperimetric inequality we have \begin{equation}\label{isoine}4\pi (\pi-\mu(t))\le |\partial_\circ A_t|^2.\end{equation} The equality in \eqref{isoine} is attained if and only if $\mathbb{D}\setminus A_t$ is a disk. Thus
$$4\pi (\pi-\mu(t))\le \mu'(t)(-\int_{u=t} |\nabla u(z)||\diff z|).$$
Now $|u(z)||\nabla u(z)|=-\left<\nabla \log u, \mathbf{n}\right>$, where $\mathbf{n}$ is the outer normal at $\partial_\circ A_t$. Therefore $$\int_{u=t} |\nabla u||\diff z|= -t\int_{u=t}\left<\nabla \log u, \mathbf{n}\right>|\diff z|.$$ Hence $$4\pi (\pi-\mu(t))\le -t\mu'(t)(-\int_{u=t}\left<\nabla \log u, \mathbf{n}\right> |\diff z|),$$  where $\mathbf{n}$ is the outer normal at $\partial A_t$.
Now we state the Green formula \begin{equation}\label{green}\int_{|z|=1}\left<\nabla \log u, \mathbf{n}\right>|\diff z|+\int_{u=t}\left<\nabla \log u, \mathbf{n}\right>|\diff z| =\int_{A_t} \Delta \log u \diff x \diff y.\end{equation}
By using \eqref{green} we have
$$4\pi (\pi-\mu(t))\le -t\mu'(t)\left(\int_{|z|=1}\left<\nabla \log u, \mathbf{n}\right>|\diff z|-\int_{A_t} \Delta \log u \diff x\diff y\right).$$
Since $\Delta \log u\ge 0$, then
$$4\pi (\pi-\mu(t))\le -t\mu'(t)\int_{|z|=1}\left<\nabla \log u, \mathbf{n}\right>|\diff z|.$$
For $|z|=1$ we have  $u(z) =1$, and then in view of \eqref{cohold} and Lemma~\ref{ulu}
\[\begin{split}
\int_{|z|=1}\left<\nabla \log u, \mathbf{n}\right>&=\int_{|z|=1}\frac{\left<\nabla u(z), \mathbf{n}\right>}{u(z)}|\diff z|\\
&=\int_{|z|=1}\left<\nabla u(z), \mathbf{n}\right>|\diff z|\\&= \int_{|z|=1}\frac{\partial u}{\partial r}|\diff z|\\&\le\int_{|z|=1}\left|\frac{\partial f(z)}{\partial r}\right||\diff z|\\&\le  \int_0^{2\pi}K\left|\frac{\partial f(e^{is})}{\partial s}\right|\diff s=   2\pi K d.\end{split}\]
Hence $$4\pi (\pi-\mu(t))\le -2t\mu'(t)\pi K d$$
i.e. \begin{equation}\label{mumono}2(\pi-\mu(t))\le -t \mu'(t) Kd.\end{equation}
Let $\lambda(s)=\mu^{-1}(s)$. Then $$2(\pi-s)\le -\frac{K d \lambda(s) }{\lambda'(s)}$$ or what is the same
$$-\frac{\lambda'(s)}{\lambda(s)}\le \frac{Kd}{2(\pi -s)}.$$ By integrating in the interval $[0, s]$, the previous inequality  for $s<\pi$, having in mind the condition $\mu(1)=0$, i.e. $\lambda(0)=1$ we obtain
$$\log \frac{1}{\lambda(s)}\le \frac{Kd}{2}\log\frac{\pi}{\pi-s}.$$
Therefore   $$\frac{1}{\lambda(s)^{2/(Kd)}}\le \frac{\pi}{\pi-s}.$$ Hence
$$\pi - s\le \pi \lambda^{2/(Kd)}(s).$$ Thus $$\mu(t)\ge \pi(1-t^{2/(Kd)}).$$
Since the distribution of $f(z)=e^{ids} |z|^{Kd}$ is $\mu_K(t)=\pi - \pi t^{2/(Kd)}$,
we get
$$\mu(t)\ge \mu_K(t).$$


\end{proof}

\begin{corollary}
Under conditions of Theorem~\ref{teo2}, let $\nu(t)=|\{z: |f(z)|<t\}|$. Then $\nu(t)\le \pi t^{2/(Kd)}$, with an equality if and only if $f(z)=e^{ids}|z|^{Kd}$. \end{corollary}

\begin{example}\label{exa}
Let $f(z) = \frac{z+a}{1+z\bar {a}}.$ Then for $0<t<1$ define $A_t=\{z: |f(z)|>t\}.$ Let $\mu_a(t)=|A_t|$.
Then by taking the change of variables $w=f(z)=r e^{i\tau}$, having in mind the  equations $$g(w) = f^{-1}(w)=\frac{w-a}{1-w\bar a}, \ \ \ |g'(w)|^2=\frac{(1-|a|^2)^2}{|1-w\bar a|^4},$$  for $z=x+i y$, we obtain
\[\begin{split} \mu_a(t)&=\int_{|f(z)|>t} \diff x \diff y
\\&=\int_{0}^{2\pi} \int_t^1 r \frac{(1-|a|^2)^2}{|1-w\bar{a}|^4}\diff r \diff \tau.
\end{split}
\]

Now we use the Parseval's formula to obtain  that  \[\begin{split}\int_{0}^{2\pi} \int_t^1 r \frac{(1-|a|^2)^2}{|1-w\bar{a}|^4}\diff r \diff \tau &=2\pi (1-|a|^2)^2\sum_{k=0}^\infty \int_t^1 (1+m)^2|a|^{2m}r^{2m+1}\diff r\\&=\frac{\pi  \left(1-t^2\right) \left(1-|a|^4 t^2\right)}{\left(1-|a|^2 t^2\right)^2}.  \end{split}\]
If we denote $r=|a|$, and $\phi(r)=\frac{\pi  \left(1-t^2\right) \left(1-r^4 t^2\right)}{\left(1-r^2 t^2\right)^2}$, then we obtain $$\phi'(r)=\frac{4 \pi  r \left(1-r^2\right) t^2 \left(1-t^2\right)}{\left(1-r^2 t^2\right)^3}$$ which is clearly a non-negative function. In particular $\phi$ is increasing and thus $\mu_0(t)\le \mu_a(t)$ which confirms our theorem.
\end{example}

\begin{corollary}
Let $f$ be a real analytic mapping of the unit disk onto itself which is $K-$quasi-regular at the boundary having a log-subharmonic modulus.  Assume also that $d=\mathrm{Ind}_{f(\mathbf{T})}(0)$. Then for $p\ge 1$,
\begin{equation}\label{ineqp1}\|f\|_{p}\ge\left(\frac{2\pi}{2 + K dp}\right)^{1/p}.\end{equation} 

The equality in \eqref{ineqp1} is attained for $f_K(z)=e^{id s } |z|^{Kd}$.
\end{corollary}

\begin{proof} Let $\mu(t)=\{z: |f(z)|\ge t\}$.
We first have $$\|f\|_p^p=p\int_0^1\mu(t) t^{p-1}\diff t.$$ Then from Theorem~\ref{teo2}, $\mu(t)\ge \mu_K(t)=\pi(1-t^{2/(K d)})$, which implies the inequality $$\|f\|^p_p
\ge \|f_K\|^p_p=p\int_0^1\mu_K(t) t^{p-1}\diff t=\frac{2\pi}{2 + K dp}. $$ This finishes the proof.
\end{proof}

In this example is shown that among the M\"obius transformation the identity has the smallest Hilbert norm. 
\begin{example}
Let $f(z) = \frac{z+a}{1+z\bar {a}}.$  Then by  Example~\ref{exa}, we calculate that  \[\begin{split}\|f\|_p^p&=p\int_0^1\mu(t) t^{p-1} \diff t\\&=p\int_0^1\frac{\pi  \left(1-t^2\right) \left(1-|a|^4 t^2\right)}{\left(1-|a|^2 t^2\right)^2} t^{p-1} \diff t.  \end{split}\]

Now for $p=2$ we have $$h(a):=\|f\|_2=\sqrt{\pi } \sqrt{\frac{2 |a|^4-|a|^2-\left(1-|a|^2\right)^2 \log\left[1-|a|^2\right]}{|a|^4}}.$$

Its graphic is given in the Figure~1.

\begin{figure}
\centering
\includegraphics{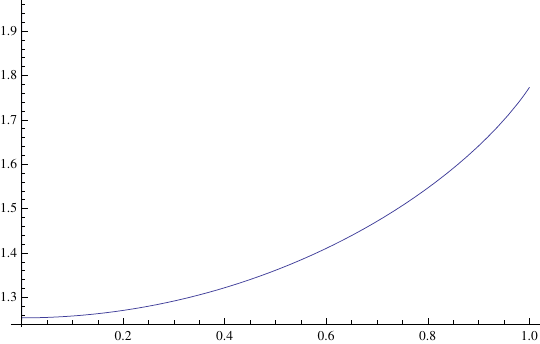}
\caption{The minimum of the function $h(|a|)$ is precisely when $a=0$ and it is equal to $\sqrt{\frac{\pi }{2}}\approx 1.25$. }
\label{21}
\end{figure}

\end{example}
In this example, a large class of radial mappings satisfying the conditions of our theorem is constructed
\begin{example} Let $f(z)= g(r) e^{is}$, where $z=r e^{is}$ and
$$g(s)=\exp\left(-\int_s^1 \left(\frac{a}{u}-\frac{\int _u^1\frac{1}{4} h(v)\diff v}{u}\right) \, \diff u\right).$$ Then $$\Delta \log |f(z)| = h(r).$$
Then $g(1)=1$. Moreover  $f$ is $K-$quasiconformal, where $$K\ge \max_{s}\frac{s g'(s)}{g(s)}=\max_s\left(a-\int_{s}^1 \frac{h(v)}{4}\diff v\right)=a.$$ Moreover $f$ maps the unit disk $\mathbb{U}$ onto itself if $$ {a}-{\int _0^1\frac{1}{4} h(v)\diff v}>0.$$

So for $0<a\le K$, and assuming that $h$ is an integrable positive function satisfying the condition  $$\int_0^1 h(s)\diff s<4a,$$ $f$ is a $K$ quasiconformal mapping so that $\Delta \log |f(z)|$ is subharmonic and satisfies the conditions of our theorem.
\end{example}

\begin{corollary}\label{qoro}
Assume that $f$ is a holomorphic mapping of the unit disk into itself, so that $\lim_{|z|\to 1}|f(z)|=1$. For $t\in[0,1]$ let $\nu(t)=|\{z: |f(z)|<t\}|.$ Then $\nu(t)\le \pi t^{2/d}$, where $d$ is the number of zeros of $f$ inside $\mathbb{D}$ including multiplicity. The equality in the last inequality is attained  for some number $t\in(0,1)$ if and only if $f(z) = e^{i\tau} z^d$ and in that case the equality is attained for every $t\in[0,1]$.

In other words  \begin{equation}\label{blaseq}|\{z:|B(z)|<t\}|\le \pi t^{2/d},\end{equation} where $$B(z)=e ^{is}\prod_{k=1}^d \frac{z-a_k}{1-z \overline{a_k}},$$ with an equality in \eqref{blaseq}  for a single $t\in(0,1)$ if and only if $a_k=0$ for every $k$.
\end{corollary}
 \begin{proof}[Proof of Corollary~\ref{qoro}]
 First of all, if $f$ is a holomorphic function mapping the unit disk onto itself such that $\lim_{|z|\to 1}|f(z)|=1$, then, by a result of Fatou \cite{fatou}, $f$ is a finite Blaschke product. Moreover, \eqref{blaseq} is a consequence of Theorem~\ref{teo2}. If, for some $t\neq 0,1$, we have $|\{z:|B(z)|<t\}|= \pi t^{2/d}$, then equality holds in the isoperimetric inequality \eqref{isoine}. In other words, $\{z:|B(z)|=t\}$ is a circle $|z-z_0|= \rho=t^{1/d}$. Therefore,
$$
|e^{is}\prod_{k=1}^d \frac{\rho e^{i\tau}-a_k}{1-\rho e^{i\tau} \overline{a_k}}|=t, \quad \tau\in[0,2\pi].
$$
Thus, the harmonic function
$$
u(z)=\log \left|B(z)\cdot \left(\frac{1-z\overline{z_0}}{z-z_0}\right)^d\right|
$$
is well-defined in the set $A_t=\{z: t<|B(z)|<1\}$ because $z_k \not\in A_t$ for every $k=0,\dots, d$. Recall that $B(a_k)=0$ for $k=1,\dots, d$ and $z_0$ is a point that is not in $A_t$, because if this were the case, $A_t$ would coincide with the unit disk, which is a contradiction. Thus, $u$ is equal to zero on the boundary $\partial A_t=\mathbf{T}\cup \partial_\circ A_t$. By the maximum principle applied to $u$ and $-u$, we conclude that $u\equiv 0$. In particular,
$$
B(z) = c \left(\frac{z-z_0}{1-z\overline{z_0}}\right)^d, \quad |c|=1.
$$
Since $|B(z)|=t$ if and only if $|z-z_0|=t^{1/d}$, we conclude that $z_0=0$.
 \end{proof}
The following example shows the evolution of sublevel set of a Blaschke product when $t$ becomes smaller.
\begin{figure}[H]
\centering
\includegraphics{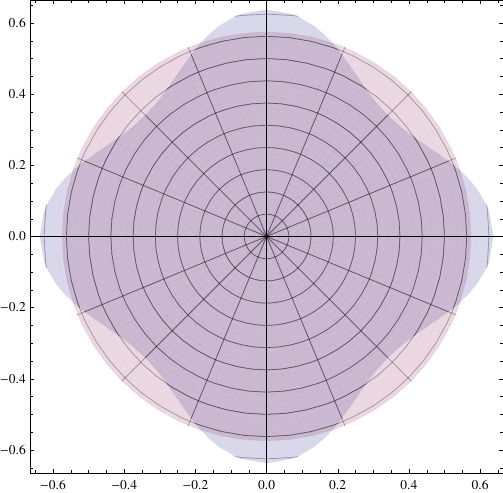}
\caption{The sublevel set of the Blaschke product  $g(z)= z^4$ for $t=1/8$ is a disk and of the Blaschke product $f(z) =\frac{1-16 z^4}{z^4-16}$ is a domain which has darker color (and smaller area). }
\label{21}

\centering
\includegraphics{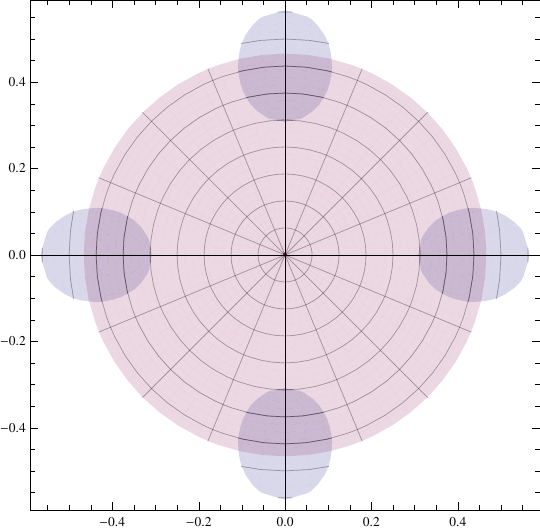}
\caption{The sublevel set of the Blaschke products  $g(z)= z^4$ for $t=1/18$ is a disk and of the Blaschke product $f(z) =\frac{1-16 z^4}{z^4-16}$ is contained of 4 small domains. }
\label{21}
\end{figure}



\end{document}